\documentclass[12pt]{article}
\usepackage[a4paper,left=2.5cm,right=2.5cm,top=3cm,bottom=3cm]{geometry}
\usepackage[parfill]{parskip}
\usepackage[utf8]{inputenc}
\usepackage{amsthm} 
\usepackage{amsmath,mathtools}
\usepackage{amssymb}
\usepackage{mathtools}
\usepackage{changepage}  
\usepackage[hyperfootnotes=false]{hyperref}
\usepackage{soul,xcolor} 
\usepackage[nottoc,notlof,notlot]{tocbibind}

\theoremstyle{definition}
\newtheorem{ex}{Type}

\makeatletter
\newtheorem*{rep@theorem}{\rep@title}
\newcommand{\newreptheorem}[2]{%
\newenvironment{rep#1}[1]{%
 \def\rep@title{#2 \ref{##1}}%
 \begin{rep@theorem}}%
 {\end{rep@theorem}}}
\makeatother

\newtheoremstyle{theorem}
{10pt}
{3pt}
{\itshape} 
{} 
{\bfseries} 
{.} 
{.5em} 
{} 

\theoremstyle{theorem}
\newtheorem{defn}{Definition}[section]
\newtheorem{lem}[defn]{Lemma}
\newtheorem{thm}[defn]{Theorem}
\newreptheorem{thm}{Theorem}

\newtheorem{cor}[defn]{Corollary}
\newtheorem{question}[defn]{Question}

\makeatletter
\def\blfootnote{\gdef\@thefnmark{}\@footnotetext}
\makeatother

\usepackage[title]{appendix}

\title{Shattering $k$-sets with Permutations}
\author{J. Robert Johnson$^{\ast}$ \and Belinda Wickes$^{\ast}$}
\date{April 2023}

\begin{document}

\maketitle
\blfootnote{$^{\ast}$ School of Mathematical Science, Queen Mary University of London, London, E1 4NS, UK. \\Contact: r.johnson@qmul.ac.uk and b.wickes@qmul.ac.uk}
\begin{abstract}
    Many concepts from extremal set theory have analogues for families of permutations. This paper is concerned with the notion of shattering for permutations. A family $\mathcal{P}$ of permutations of an $n$-element set $X$ \textit{shatters} a $k$-set from $X$ if it appears in each of the $k!$ possible orders in some permutation in $\mathcal{P}$. The smallest family $\mathcal{P}$ which shatters every $k$-subset of $X$ is known to have size $\Theta(\log n)$.

    Our aim is to introduce and study two natural partial versions of this shattering problem.

    Our first main result concerns the case where our family must contain only $t$ out of $k!$ of the possible orders. When $k=3$ we show that there are three distinct regimes depending on $t$: constant, $\Theta(\log\log n)$, $\Theta(\log n)$. We also show that for larger $k$ these same regimes exist although they may not cover all values of $t$.

    Our second direction concerns the problem of determining the largest number of $k$-sets that can be totally shattered by a family with given size. We show that for any $n$, a family of $6$ permutations is enough to shatter a proportion between $\frac{17}{42}$ and $\frac{11}{14}$ of all triples. 
\end{abstract}


\section{Introduction}

\subsection{Background}\label{intro pt 1}
Let $S_n$ be the set of all permutations of $\{1,\dotsc,n\}$ thought of as ordered $n$-tuples. Our aim is to study properties of families of permutations from $S_n$ inspired by concepts of shattering from extremal set theory. We begin with the notion of shattering for sets. Let $\mathcal{F}$ be a family of subsets of $[n]=\{ 1,2,\dotsc,n \}$ and let $A \subseteq [n]$, we say that $A$ is \textit{shattered} by $\mathcal{F}$ if for each $B \subseteq A$ there exists a set $S \in \mathcal{F}$ such that $A \cap S = B$. The notion of shattered sets has uses throughout combinatorics and computer science, with the focus being on the size of families that shatter certain sets. For examples of work in this area see \cite{exp construction}, \cite{indep sets}, \cite{Sauer}, \cite{Shelah}, and \cite{VC}.

A family $\mathcal{R}$ of subsets of $[n]$ is \textit{$k$-independent} if any $R_1,R_2,\dotsc,R_k \in \mathcal{R}$ have the property that all $2^k$ intersections $\cap_{i=1}^k J_i$ are non-empty, where $J_i$ takes on either $R_i$ or its complement $R_i^c$. A large $k$-independent family $\mathcal{R}$ gives rise to a small family $\mathcal{F}$, where $\mathcal{F} \subseteq 2^{[r]}$ is a family that shatters all the $k$-subsets of $[r]$ and $|\mathcal{R}|=r$. To see this, first note that for each $x \in [n]$ we can define a set $F(x) \subseteq [r]$ by setting $i \in F(x)$ if and only if $x \in R_i$. Since $\mathcal{R}$ is $k$-independent, a family consisting of all such $F(x)$ sets must shatter all $k$-subsets of $[r]$. Set $\mathcal{F} = \{ F(x) : x \in [n] \}$ then clearly $|\mathcal{F}|=n$, so the bigger $r$ is, the greater the number of $k$-subsets shattered by $n$ sets. Kleitman and Spencer \cite{indep sets} posed the question `How large can a family of $k$-independent sets be?' which is equivalent to the question `How small can a family that shatters every $k$-subset of $[n]$ be?'.

\begin{thm}\label{indep set bound} {\normalfont(Kleitman and Spencer \cite{indep sets})}
    For fixed $k$ and $n$ sufficiently large there are absolute constants $d_1$ and $d_2$ such that
    \begin{align*}
    2^k d_2\log n \leq g_k(n) \leq k2^k d_1\log n
\end{align*}
where $g_k(n)$ is the size of the smallest family from $2^{[n]}$ that shatters every $k$-subset of $[n]$.
\end{thm}

Our starting point is an analogue of this family size problem using permutations in place of sets. This has been studied under a variety of names (\cite{furedi}, \cite{total shattering lower bound},  \cite{Tarui}, \cite{spencer: first element}): completely scrambling permutation families, mixing permutations, and sequence covering arrays. We first formalise the problem, establish notation, and summarise the previous work.

Consider the set $S_n$ of all permutations of $[n]$, any permutation $P\in S_n$ corresponds to a particular linear order of the elements of $[n]$, so $P$ can be written as $P = (p_1,p_2,\dotsc,p_n)$ where $\{p_1,p_2,\dotsc,p_n \} = [n]$. Note that this is not cycle notation, rather it can be thought of as the second line of two-row notation. For all $a,b \in [n]$, we write $a <_P b$ to mean that $a$ precedes $b$ in the permutation $P \in S_n$.

Suppose $R$ is a permutation of $[k]$. For a $k$-tuple $\{ a_1,\dotsc,a_k \} \subseteq [n]$ with $a_1 < \dots < a_k$ and a $P \in S_n$, we say $\{ a_1,\dotsc,a_k \}$ follows the pattern $R \in S_k$ in $P$ if $a_i <_P a_j$ for all $i,j \in [k]$ with $i <_R j$. That is, $\{ a_1,\dotsc,a_k \}$ follows the pattern $R$ in $P$ if the restriction of $P$ to $\{ a_1,\dotsc,a_k \}$ is order isomorphic to $R$. For any $k$-tuple $X$ from $[n]$, we use $P_X$ to denote the permutation pattern from $S_k$ followed by $X$ in $P \in S_n$. We can express our shattering condition through permutation patterns.

\begin{defn}\label{def shatter}
    We say that a family $\mathcal{S} \subseteq S_n$  \textit{shatters} the $k$-tuple $X \subseteq [n]$ if $\{ P_X : P \in \mathcal{S} \}=S_k$.
\end{defn}
In other words, $\mathcal{S}$ shatters $X$ if every possible ordering of the elements of $X$ appears in the permutations of $\mathcal{S}$.

We are interested in families that shatter many different sets at the same time, in particular when all sets of the same size are shattered by one family.
Let $f_k(n)$ be the smallest integer such that there exists a family $\mathcal{S}$ of permutations from $S_n$ that shatters every $k$-tuple from $[n]$ and has $|\mathcal{S}|=f_k(n)$. Throughout, if a family is said to shatter every $k$-tuple in $[n]$, then it is implied that said family is a subset of $S_n$.

Clearly we have that $f_k(n) \geq |S_k|$, otherwise we certainly cannot shatter any $k$-tuple. It is also plain that $f_k(k)=k!$ since the family $S_k$ is suitable. 

The order of magnitude of the value $f_k(n)$ is known asymptotically. The upper bound given by Spencer \cite{spencer: first element} can be seen with a simple probabilistic argument, the lower bound is more involved and was shown by Radhakrishnan \cite{total shattering lower bound}
\begin{equation*}
    \left(\frac{(k-1)!}{\log e} + o(1)\right)\log n \leq f_k(n) \leq \frac{k }{\log k! - \log (k!-1)}\log n.
\end{equation*}

To simplify the notation throughout, $\log (n) = \log_2(n)$ unless otherwise stated.

In the case where $k=3$ the best upper bound is actually given by a construction rather than using the probabilistic method. The construction is given by Tarui in \cite{Tarui} and the lower bound from F\"{u}redi \cite{furedi}
\begin{equation*}
   \frac{2 }{\log e}\log n \leq f_3(n) \leq 2\log n + (1+o(1))\log\log n.
\end{equation*}
For $k>3$ no explicit construction that matches the order of magnitude of the probabilistic upper bound is known.

Our main focus will be partial and fractional variations of this problem, which we introduce in the next section.

\subsection{Partial and Fractional Shattering Problems}\label{intro pt 2}
As well as the problem of determining the smallest size of a permutation family that shatters every $k$-tuple, it is natural to ask about small families that cover some subset of orders for every $k$-tuple. Equally it is natural to consider small families that shatter some partial collection of $k$-tuples. Our main aim is to introduce and investigate these variations of the original family size problem. There are two different problems we will consider.
\begin{itemize}
    \item Given $t \leq k!$, find the smallest family of permutations from $S_n$ which ensures each $k$-tuple appears in at least $t$ orders.
    \item Given $\alpha \in [0,1]$, find the smallest family that shatters at least $\alpha \binom{n}{k}$ of all $k$-tuples.
\end{itemize}

Note that there is another possible definition of partial shattering. The problem is, given some fixed set of patterns $T \subseteq S_k$, find the smallest family of permutations from $S_n$ such that every $k$-tuple follows all of the patterns in $T$. It turns out that for any $\mathcal{S} \subseteq S_n$ in which all $k$-tuples follow a specific non-monotone pattern, we have $|\mathcal{S}|=O(\log n)$ (see Lemma \ref{T lower}). This matches the lower bound for total shattering, making this variation uninteresting.

For our first problem we formally define partial shattering as follows, using the same notation as Definition \ref{def shatter}.
\begin{defn}\label{def partial}
    We say that a family $\mathcal{S} \subseteq S_n$ (partially) $t$-shatters the $k$-tuple $X \subseteq [n]$ if $|\{ P_X : P \in \mathcal{S} \}|\geq t$. We define $f_k(n,t)$ to be the smallest integer such that there exists a family $\mathcal{S}$ that $t$-shatters every $k$-tuple in $[n]$, and $|\mathcal{S}|=f_k(n,t)$.
\end{defn}
Clearly $f_k(n,k!)$ is the size of the smallest shattering family on $k$-tuples, $f_k(n,k!)=f_k(n)$. We also have the trivial cases $f_k(n,1)=1$ and $f_k(n,2)=2$ which can be seen by taking only monotone permutations.

Spencer \cite{spencer: first element} and F\"{u}redi \cite{furedi} also discuss variations which look to cover a subset of patterns of each $k$-tuple where each element appears in a specific place. Neither of these variations can be expressed in terms of partial $t$-shattering or vice versa. In \cite{furedi}, F\"{u}redi defines an extremely general framework called $\mathcal{S}$-mixing, we note that partial $t$-shattering is one of the most natural special instances of this $\mathcal{S}$-mixing.

For our second problem it is more natural to ask the question in reverse than it is to fix the fraction. We ask for the maximum number of shattered $k$-tuples from a family of fixed size.
\begin{defn}\label{def fractional}
    Let $F_k(n,m)$ be the largest $\alpha \in [0,1]$ such that there exists a collection $\mathcal{R}$ of exactly $m$ permutations from $S_n$, with the property that $\alpha \binom{n}{k}$ $k$-tuples are totally shattered by $\mathcal{R}$. We call this fractional shattering of $n$ with $m$ permutations.
\end{defn}
It is plain that $F_k(n,m) = 0$ whenever $m < k!$ and that $F_k(k,k!)=1$. In fact, by a result of Levenshtein \cite{perfect family} we have that $F_k(k+1,k!)=1$.

This instance of shattering is genuinely different from Partial Shattering, and showcases different behaviour as a result. Unlike Partial Shattering where the constructed families grow with $n$, the families constructed for Fractional Shattering only contain a constant number of permutations.

The remainder of this paper will be structured as follows. Starting with partial shattering we begin by showing that fixing a non-monotone pattern which all $k$-tuples follow results in a large ($\Theta(\log n)$) family of permutations. We then consider the case where $k=3$ and classify the size of $f_3(n,t)$ asymptotically for all values of $t$.
\begin{thm}\label{class t 3}
We have the following bounds on $f_3(n,t)$
\begin{equation*}
    f_3(n,t) = 
    \begin{cases}
    t & \text{for } t=1,2 \\
    \Theta( \log\log n )  & \text{for } t=3,4 \\
    \Theta( \log n )  & \text{for } t=5,6.
    \end{cases}
\end{equation*}
\end{thm}

We follow up with the extension to larger values of $k$. The same separation into three distinct size categories follows over to the $k>3$ cases, although in this setting there is the possibility that another size class exists for $t \in [k+1,2(k-1)!]$.

\begin{thm}\label{class t k}
We have the following bounds on $f_k(n,t)$
    \begin{equation*}
    f_k(n,t) = \begin{cases}
    t  & \quad \text{ when } t=1,2\\
    \Theta(\log\log n) & \quad \text{ when } t \in [3,k]\\
    \Theta(\log n) & \quad \text{ when } t \in [2(k-1)!+1,k!].
    \end{cases}
\end{equation*}
\end{thm}

A natural question is then, whether or not $f_k(n,t)$ always one of the three sizes we see here or if there is another class. This highlights the interesting question of how many size classes there are for such shattering problems in general. We note that for all the permutation shattering variants discussed above and in \cite{furedi, total shattering lower bound, Tarui, spencer: first element}, the smallest known families realising each has size $\Theta(\log n)$, $\Theta(\log\log n)$, or constant.

We then move on to the fractional version of the problem. Again we focus on the case where $k=3$, the first interesting case is $F_3(n,6)$ where we get the following bounds.
\begin{thm}\label{frac: 3 6}
For any $n\geq 8$ we have
\begin{equation*}
    \frac{17}{42} \leq F_3(n,6) \leq \frac{11}{14}.
\end{equation*}
\end{thm}
In fact our method gives a slightly stronger upper bound which is hard to quantify but shows that the upper bound is in fact strict when $n>8$. We also show that in general the value of $F_k(n,m)$ is decreasing in $n$.
\begin{thm}\label{frac: weakly decreasing}
For fixed $k$ and fixed $m\geq k!$ we have $F_k(n,m) \geq F_k(n+1,m)$.
\end{thm}
This means that the limit for $F_k(n,m)$ exists for fixed $k$, $m$ and as $n$ tends to infinity, we see that this limit lies strictly between $0$ and $1$ for all $k$ and $m$. 

We also look at a method of iterating small families in such a way as to give a family of permutations on much larger $n$, that preserves most of the shattering conditions from the initial family. We give some bounds using this method and some small starting families we call Perfect Families. A perfect family for $k$ on $n$, denoted $\mathcal{Q}_k(n)$, is a family of exactly $k!$ permutations from $S_n$ which shatter every $k$-tuple in $[n]$ (a realisation of the family giving $F_k(n,k!)=1$). An example of a family \hypertarget{Q3(4)}{$\mathcal{Q}_3(4)$} is given by the following:
\begin{equation*}
    \begin{array}{ l l l}
        Q_1= (1,2,3,4) \quad& Q_2= (2,4,1,3) \quad& Q_3= (3,4,1,2) \\
        Q_4= (1,4,3,2) \quad& Q_5= (4,2,3,1) \quad& Q_6= (3,2,1,4).
    \end{array}
\end{equation*}
We note that Levenshtein's result in \cite{perfect family} can be expressed as showing a Perfect Family for $k$ on $n$ exists whenever $n=k+1$.

Finally we have a construction for the original total shattering problem, that is an upper bound for $f_k(n)$. The best bound, of order $\log n$, is given by a probabilistic argument and no explicit construction of this size is known. Our construction gives a family with a power of $\log n$ permutations, this is above the known upper bound for $f_k(n)$ but is constructive.

We finish with some open problems about all the topics covered.

\section[Partially shattering every \texorpdfstring{$k$}{k}-tuple]{Partially shattering every \texorpdfstring{$\boldsymbol{k}$}{k}-tuple}

The following are two well known but useful results that will be used throughout.

\begin{lem}\label{useful lem}
Let $(A,B)$ be a partition of $[n]$, so $A \cup B = [n]$ and $A\cap B = \emptyset$. Any family $\mathcal{U}$ of such partitions with the property that, for every $x,y \in [n]$ there exists $(A,B) \in \mathcal{U}$ where exactly one of $x$ and $y$ is in $A$ and the other is in $B$, also satisfies 
$|\mathcal{U}| \geq \lceil \log n \rceil$. Furthermore there exists such a family $\mathcal{U}$ where the bound holds with equality.
\end{lem}

\begin{lem}\label{useful lem 2} {\normalfont(Chung, Graham, and Winkler \cite{addressing})}
Let $(A,B)$ be a partition of $[n]$, so $A \cup B = [n]$ and $A\cap B = \emptyset$. Any family $\mathcal{U}$ of these partitions with the property that, for every $x,y \in [n]$ there exists $(A,B) \in \mathcal{U}$ where $x \in A$ and $y \in B$, must satisfy $|\mathcal{U}| \geq \lceil \log n + \left(\frac{1}{2} + o(1)\right)\log\log n \rceil$. Moreover, there exists such a family with $|\mathcal{U}| = \lceil \log n + \left(\frac{1}{2} + o(1)\right)\log\log n \rceil$.
\end{lem}

Next we look at a useful lemma showing that whenever we require a family of permutations to follow a fixed non-monotone pattern, that family has the same order of magnitude as a totally shattering family. This is the reason we have chosen to define partially shattering as in Definition \ref{def partial}.

\begin{lem} \label{T lower}
Let $n \geq 3$ and $R \in S_k$ be any non-monotone permutation pattern. If $\mathcal{S} \subseteq S_n$ is a family of permutations for which every $k$-tuple follows $R$ in at least one $P\in \mathcal{S}$, then we must have $|\mathcal{S}| \geq \log(n-k+2)$. Hence $|\mathcal{S}| = \Omega(\log n)$.
\end{lem}

\begin{proof}
Let $\mathcal{S}$ and $R$ be as described in the statement of the lemma. Note that whenever $R$ is non-monotone there is some element $x \in [k]$ such that when $R$ is restricted to $\{ x, x+1, x+2 \}$ the triple is non-monotone.

If $R$ induces $(x+1,x,x+2)$ or $(x+2,x,x+1)$ then set $y = x$. If $R$ induces $(x,x+2,x+1)$ or $(x+1,x+2,x)$ then set $y=n - k +x+2$.

For each $P \in \mathcal{S}$ we generate a partition of the set $W:=[n] \setminus ([x-1] \cup [n - k +x+3,n] \cup y)$ into two parts $A_P$ and $B_P$, where $A_P$ contains every $w\in W$ such that $y<_P w$ and $B_P = W \setminus A_P$.

Note that when $y=x$ the set $W$ contains only elements larger than or equal to $y$, and when $y=n - k +x+2$ the set $W$ only contains smaller or equal elements. Hence for any pair $a,b \in W$ we must have one of the orders $(a,y,b)$ or $(b,y,a)$ appearing in $\mathcal{S}$ since exactly one of them follows $R$ as part of the $k$-tuple $\{ 1,2,\dotsc,x-1,y,a,b,n-k+x+3,\dotsc,n \}$. To see this note that the elements that correspond to $\{ x,x+1,x+2 \}$ are exactly $\{ y,a,b \}$.

Then we have satisfied the conditions for Lemma \ref{useful lem} and we must have that $|\mathcal{S}| \geq \log |W|$, which gives the result.
\end{proof}

Our aim will be to prove the following classification of $f_3(n,t)$.

\begin{repthm}{class t 3}
We have the following bounds on $f_3(n,t)$
\begin{equation*}
    f_3(n,t) = 
    \begin{cases}
    t & \text{for } t=1,2 \\
    \Theta( \log\log n)  & \text{for } t=3,4 \\
    \Theta( \log n)  & \text{for } t=5,6.
    \end{cases}
\end{equation*}
\end{repthm}
The upper bound when $t=6$ comes from the total shattering bounds for $f_3(n)$. The value of $f_3(t)$ whenever $t=1,2$ is trivial. Indeed note that any one single permutation requires each triple to follow some pattern (not necessarily the same pattern). So we get that $f_3(n,1)=1$ simply by choosing any $P \in S_n$. We call $\overline{P}$ the \textit{reverse} permutation of $P$ if $a <_{\overline{P}} b$ whenever $b <_P a$. All triples follow a different pattern in $\overline{P}$ as they do in $P$, therefore we must have that $f_3(n,2)=2$. These observations are not limited to triples and can be applies identically when $k>3$.

To get a lower bound when $t \geq 3$ we need to work a little harder, but it does not complicate the method to consider the general $k$-tuples rather than just triples. Observe that taking $P$ to be the increasing permutation $(1,2,\dotsc,n-1,n)$ we get that $\overline{P}$ must be the decreasing permutation. Now consider a third permutation $P'$, any $k$-tuple contained in a monotone subpermutation of $P'$ will not follow a new pattern. We know by the Erd\H{o}s-Szekeres Theorem that when $n$ is large we must have some reasonably large monotone subpermutation. We use this idea to get the lower bound in this case.

\begin{thm}{\normalfont(Erd\H{o}s-Szekeres Theorem \cite{ES thm})}\label{ES thm}
Let $r,s \in \mathbb{N}$, then any sequence of real numbers with length at least $n=rs+1$ contains an increasing subsequence of length at least $r+1$ or a decreasing subsequence with length at least $s+1$.
\end{thm}

\begin{thm}\label{ES lower}
For any $n \geq 3$ and every $t\geq 3$, we have $f_k(n,t) \geq  \log\log n - C$ where $C$ is a constant dependant on $k$ and $t$. More precisely we have
\begin{equation*}
    f_k(n,t) \geq \log\log(n-1) - \log\log(k-1) + t -3.
\end{equation*}
\end{thm}
\begin{proof}
Let $\mathcal{S}$ be a family of permutations of $[n]$ that $t$-shatters every $k$-tuple. Suppose for a contradiction that $|\mathcal{S}| \leq \log\log(n-1) - \log\log(k-1) + t -4$. 

Choose any $t-3$ permutations from $\mathcal{S}$, and set $\mathcal{S'}$ to be the $\mathcal{S}$ with the chosen permutations removed. This is a family of permutations from $S_n$ that $3$-shatters every $k$-tuple. Note that $|\mathcal{S'}| \leq \log\log(n-1) - \log\log(k-1) -1 =m$.

Take any $P_1 \in \mathcal{S'}$, then by the \hyperref[ES thm]{Erd\H{o}s-Szekeres Theorem} $P_1$ must contain an increasing subsequence of length $r=\lfloor (n-1)^\frac{1}{2} \rfloor +1$ (or a decreasing subsequence of length $r$). Let the elements in this monotone subsequence be written as $X_1 = \{x_1,\dotsc,x_r\}$, and let $P_1(X_1)$ be the restriction of $P_1$ to the elements of $X_1$. Then $P_1(X_1)$ is simply a permutation of $X_1$ where the elements appear in the same order that they appear in $P_1$.

Look at another permutation $P_2 \in \mathcal{S'}$ restricted to the elements of $X_1$, $P_2(X_1)$. Applying Erd\H{o}s-Szekeres again, this time to $P_2(X_1)$, we see that there must be an monotonic subsequence of length $\lfloor (r-1)^\frac{1}{2} \rfloor +1$. Let $X_2 \subseteq X_1$ be the set of elements in this subsequence. Then consider $P_3(X_2)$ and generate $X_3 \subseteq X_2$ in an analogous manner.

Take each permutation from $\mathcal{S'}$ into consideration one by one, at  step $i$ generate a set of `bad' elements $X_i \subseteq X_{i-1}$ by applying Erd\H{o}s-Szekeres to $P_i(X_{i-1})$ and finding a monotone subsequence of length at least $\lfloor (|X_{i-1}|-1)^\frac{1}{2} \rfloor +1$.

Consider the set $X_m$, it must contain elements that appear in a monotone subsequence of every permutation in $\mathcal{S'}$. In other words, the elements of $X_m$ appear in a maximum of $2$ possible orders. Therefore if $|X_m|\geq k$ then $X_m$ contains a $k$-tuple that does not appear in $3$ orders across $\mathcal{S'}$. Hence, from our initial conditions we must have $|X_m|< k$.

Note that the number of elements we are restricting to in the final stage is at most
\begin{equation*}
    (n-1)^{(\frac{1}{2})^m} + 1.
\end{equation*}
Then observe
\begin{align*}
    (n-1)^\frac{1}{2^m} + 1 &< k\\
    \log\log(n-1) - \log\log(k-1) &< m.
\end{align*}

On the other hand, based on the assumed size of $\mathcal{S}$ we have that $m=\log\log(n-1) - \log\log(k-1) -1$, a contradiction. Therefore we must have
\begin{equation*}
    |\mathcal{S}|\geq \log\log(n-1) - \log\log(k-1) + t -3.
\end{equation*}

\end{proof}

From this we can see that if $t\geq 3$ then $f_k(n,t)$ is always between $\Omega(\log\log n)$ and $O(\log n)$. In the case where $k=3$, Theorem \ref{class t 3} shows that the size of $f_3(n,t)$ always falls into one of these size categories.

The next result gives an upper bound on $f_3(n,4)$, first we show the bound in Theorem \ref{little construction bound} using the recursion of Lemma \ref{little construction}, then we give the construction that provides the recursion.

\begin{lem} \label{little construction}
For $n \geq 3$ we have that $f_3(n^n,4) \leq f_3(n,4) + \log n + \left(\frac{1}{2} + o(1)\right)\log\log n  + 2$.
\end{lem}

Therefore we get the following bound.
\begin{thm}\label{little construction bound}
For large n we have $\log\log n  \leq f_3(n,4) \leq 2 \log\log n$
\end{thm}
\begin{proof}
The lower bound is directly from Lemma \ref{ES lower} with $k=3$ and $t=4$.

For the upper bound, write $n = (m^m)^{m^m}$ for some real number $m$, noting that $m$ may not be an integer. Then from Lemma \ref{little construction} we have
\begin{align*}
    f_3(n,4) &\leq f_3(\lceil m^m \rceil,4) + \log \lceil m^m \rceil + \left(\frac{1}{2}+o(1)\right)\log \log \lceil m^m \rceil  + 2\\
    &\leq  f_3((m+1)^{m+1},4) + \log m^m + \log \log m^m + 3\\
    &\leq  f_3(m+1,4) + \log(m+1) + \log \log (m+1) + 3 + \log m^m + \log \log m^m + 3\\
    &\leq  f_3(m+1,4) + \log m  + \log \log m  + \log m^m + \log \log m^m + 8\\
    \intertext{we can now use Tarui's upper bound for $f_3(n)$ in \cite{Tarui}}
    &\leq  2(\log (m+1)  + \log \log (m+1)) + \log \log m^m  + \log \log (m^m)^{m^m} + 8\\
    &\leq  2(\log m  + \log \log m) + \log \log m^m  + \log \log (m^m)^{m^m} + 12\\
    &\leq  2\log \log m^m + \log \log m^m  + \log \log (m^m)^{m^m} + 12\\
    &\leq  \log \log n + 3\log \log m^m + 12\\
    &\leq  2\log \log n.
\end{align*}

\end{proof}

Now we see the construction that gives the recursion.

\begin{proof}[Proof of Lemma \ref{little construction}]
Let $\mathcal{S}$ be a $4$-shattering family for triples in $[n]$ with $|\mathcal{S}| = f_3(n,4)$. We will use this family to construct a new family from $S_{n^n}$ that $4$-shatters every triple.

Assign each $x \in [n^n]$ to a unique string $(x_1,\dotsc,x_n) \in [n]^n$, by equating the standard order on $[n^n]$ with the lexicographic order on $[n]^n$. That is, $x=1$ is assigned to $(1,1,\dotsc,1,1,1)$, $x=n$ is assigned $(1,1,\dotsc,1,1,n)$, $x=n+1$ is assigned $(1,1,\dotsc,1,2,1)$, and so on. We call $(x_1,\dotsc,x_n)$ the code (or unique code) for $x$. Let $d:[n^n]^2 \to [n]$ be the function giving the first coordinate that differs between two elements, so $d(x,y) = \min\{i:x_i \neq y_i\}$. We will use these unique codes to generate two types of permutations on $[n^n]$.

\begin{ex}\phantomsection\label{type 1}
\begin{adjustwidth}{0.5cm}{}
    \hspace{-0.5cm}Here we apply permutations from $\mathcal{S}$. We will generate one permutation $P^n$ of $[n^n]$ from each permutation $P \in \mathcal{S}$.
    
    Consider any $P\in \mathcal{S}$, we can apply $P$ to any set of $n$ objects. In particular we can apply $P$ to each coordinate of the unique code of every $x \in [n^n]$, call the resulting string the $P$-permuted unique code of $x$. Having found the $P$-permuted unique code of every $x \in [n^n]$, we get a permutation $P^n \in S_{n^n}$ by considering the order on $[n^n]$ given by the lexicographic order on the $P$-permuted unique codes of each $x \in [n^n]$.

    The result is that for elements $x,y \in [n^n]$ with $d(x,y)=i$, we have that $x$ precedes $y$ in $P^n$ if and only if $x_i$ precedes $y_i$ in $P$. In other words, $x <_{P^n} y$ if and only if $x_i <_P y_i$.
\end{adjustwidth}
\end{ex}

We do this for all $P\in \mathcal{S}$ which gives us $|\mathcal{S}| = f_3(n,4)$ permutations of $[n^n]$, call this collection of permutations $\mathcal{S}^n$. 

To see which triples are now $4$-shattered, consider the triple $\{x,y,z\}$ with the codes $(x_1,x_2,\dotsc,x_n)$, $(y_1,y_2,\dotsc,y_n)$ and $(z_1,z_2,\dotsc,z_n)$ respectively. Suppose all three elements agree in the first $k \in [0,n-1]$ coordinates, so $x_i = y_i = z_i$ for $i \leq k$, and further suppose that none agree in coordinate $k+1$, that is $d(x,y)=d(x,z)=d(y,z)=k+1$. Then note that the order of $x,y,z$ in $P^n$ relies only on the order of $x_{k+1},y_{k+1},z_{k+1}$ in $P$. Since $\mathcal{S}$ $4$-shatters triples in $[n]$ there must be permutations $P_1,P_2,P_3,P_4 \in \mathcal{S}$ that $4$-shatter $\{x_{k+1},y_{k+1},z_{k+1}\}$, hence $P^n_1,P^n_2,P^n_3,P^n_4$ must $4$-shatter $\{x,y,z\}$

The only triples that do not have $4$ orders covered by permutations in $\mathcal{S}^n$ are those that have two elements that agree in the first $k$ coordinates of their unique code and the final element only agrees in the first $r$ coordinates where $r < k$. This is equivalent to triples $\{x,y,z\}$ with $d(x,y)=k$ and $d(x,z)=d(y,z)=r$, we will call such triples `bad'. Note that for a `bad' triple $\{x,y,z\}$ where $x<y<z$ we must have either $d(x,y)=k$ and $d(x,z)=d(y,z)<k$ or $d(y,z)=k$ and $d(x,y)=d(x,z)<k$. It cannot be that $d(x,z)=k$ and $d(x,y)=d(y,z)=r<k$ because $x<y<z$ means that $x_r<y_r<z_r$ but $d(x,z)=k$ implies $x_r=z_r$.

We can assume without loss of generality that $\mathcal{S}$ contains the monotone increasing order, therefore we can assume $\mathcal{S}^n$ contains it. We now construct the other collection of permutations to cover orders on these `bad' triples.

\begin{ex}\phantomsection\label{type 2}
\begin{adjustwidth}{0.5cm}{}
    \hspace{-0.5cm}By Lemma \ref{useful lem 2} we are able to find $\lceil \log n + \left(\frac{1}{2} + o(1)\right)\log\log n \rceil$ partitions of $[n]$ into two sets $I$ and $D$ that satisfy the conditions in \ref{useful lem 2}. In this section we will view $[n]$ as the set of coordinates for the unique codes, and therefore $I$ and $D$ as partitions of the coordinates. For each partition we will define one permutation in $S_{n^n}$.
    
    Let $(I,D)$ be any of our partitions and start by considering the first coordinate. If $1 \in I$ then we order the elements increasing by coordinate $1$, namely $x$ will precede $y$ if $x_1 < y_1$. If $1 \in D$ then order elements with decreasing coordinates, meaning $x$ will precede $y$ if $x_1 > y_1$. Next look at the second coordinate, if $2 \in I$ order such that $x$ precedes $y$ if $x_2 < y_2$, otherwise order so $x$ precedes $y$ if $x_2 > y_2$. Continue in this manner until the elements have been ordered with respect to each of their coordinates.
    
    Formally, consider $x,y \in [n^n]$ and let $d(x,y)=i$, we have that $x$ precedes $y$ if
\begin{equation*}
    \begin{cases}
    x_i < y_i  & \quad \text{ and } i \in I\\
    x_i > y_i  & \quad\text{ and } i \in D.
    \end{cases}
\end{equation*}

We use the above process to construct $\lceil \log n +\left(\frac{1}{2} + o(1)\right)\log\log n \rceil$ permutations with the property that for any coordinates $i$ and $j$ we can always find two permutations such that one has $i$ increasing and $j$ decreasing and the other has $i$ decreasing and $j$ increasing.

\end{adjustwidth}
\end{ex}

Let $\mathcal{T}$ be the set of these \hyperref[type 2]{Type 2} permutations along with the permutation of $[n^n]$ that is totally decreasing (if this is not already included by the two types), then $|\mathcal{T}| \leq \lceil \log n + \left(\frac{1}{2} + o(1)\right)\log\log n \rceil +1$.

To identify the triples that are partially shattered by $\mathcal{T}$, consider a triple $\{x,y,z\}$ where $x<y<z$ which was not $4$-shattered by $\mathcal{S}^n$. As discussed after \hyperref[type 1]{Type 1} there are two cases, either $d(x,y)=i$ and $d(x,z)=d(y,z)=j$ where $j<i$, or $d(y,z)=i$ and $d(x,y)=d(x,z)=j$ for $j<i$.

Suppose $d(x,y)=i$ and $d(x,z)=d(y,z)=j$ where $j<i$. We know from our assumption that the order $(x,y,z)$ appears in $\mathcal{S}^n$, we also know that $(z,y,x)$ appears in $\mathcal{T}$ since we included the decreasing permutation here. Furthermore we know that there is some permutation in $\mathcal{T}$ where $i$ is increasing and $j$ is decreasing. Since $x<z$ and $d(x,z)=j$ we must have $x_j<z_j$ from the construction of the unique codes, similarly we have $x_i<y_i$. Then for any permutation with $i$ increasing, $x$ must appear before $y$, and $j$ decreasing means $x$ (and $y$ since $x_j=y_j$) comes after $z$. This means that the order $(z,x,y)$ is covered. Similarly there is a permutation where $i$ is decreasing and $j$ is increasing, giving the order $(y,x,z)$. By the same reasoning, if we are in the second case where $d(y,z)=i$ and $d(x,y)=d(x,z)=j$, then we find the orders $(x,y,z)$, $(z,y,x)$, $(y,z,x)$ and $(x,z,y)$.

We now have our desired partial shattering condition, the family $\mathcal{S}^n\cup\mathcal{T}$ covers $4$ orders for each triple in $[n^n]$.

Thus
\begin{align*}
    f_3(n^n,4) \leq |\mathcal{S}^n| + |\mathcal{T}| = f_3(n,4) + \lceil \log n + \left(\frac{1}{2} + o(1)\right)\log\log n \rceil + 1.
\end{align*}
\end{proof}

We now have all the ingredients needed to prove Theorem \ref{class t 3}.

\begin{proof}[Proof of Theorem \ref{class t 3}]
Clearly $f_3(n,1)=1$ as any $P \in S_n$ forces all triples from $[n]$ to appear in one order. Recall the reverse permutation pf $P$, $\overline{P}$, then we must have that all triples follow a different pattern in $\overline{P}$ as they did in $P$, hence $f_3(n,2)=2$.

The lower bound $f_3(n,3) \geq \log\log(n-1)$ comes directly from Theorem \ref{ES lower}. The upper bound $f_3(n,4) \leq 2\log\log n$ comes directly from Theorem \ref{little construction bound}. This gives us $f_3(n,t) = \Theta(\log\log n)$ when $t=3,4$.

To see $f_3(n,5) \geq \log(n-1)$, let $\mathcal{S}$ be a family of permutations from $S_n$ that $5$-shatters every triple. Consider any triple of the form $\{ n,x,y \}$, we must have at least one of the orders $(x,n,y)$ and $(y,n,x)$ appearing in some permutation from $\mathcal{S}$ otherwise we have at most $4$ orders for $\{ n,x,y \}$. For each $P \in \mathcal{S}$ generate a partition of $[n-1]$ by having
\begin{align*}
    A_P &:= \{ x \in [n-1]: x \text{ appears after } n \text{ in } P \} \\
    B_P &:= \{ x \in [n-1]: x \text{ appears before } n \text{ in } P \}.
\end{align*}
Then using Lemma \ref{useful lem} in order to ensure at least one of the orders $(x,n,y)$ and $(y,n,x)$ is seen we must have at least $\log(n-1)$ permutations in $\mathcal{S}$.

Finally we have that $f_3(n,6) \leq (2 +o(1))\log n$ from \cite{Tarui} since $f_3(n,6) = f_3(n)$.
\end{proof}

For triples, the different values of $t$ feed equally into the three size classifications. In the general $k$-tuple case, we actually have that for most values of $t$ we require $O(\log n)$ permutations.

\begin{repthm}{class t k}
We have the following bounds on $f_k(n,t)$
    \begin{equation*}
    f_k(n,t) = \begin{cases}
    t  & \quad \text{ when } t=1,2\\
    \Theta(\log\log n) & \quad \text{ when } t \in [3,k]\\
    \Theta(\log n) & \quad \text{ when } t \in [2(k-1)!+1,k!].
    \end{cases}
\end{equation*}
\end{repthm}

The value of $f_k(n,t)$ for $t \in [k+1,2(k-1)!]$ is unknown but does lie between $\log\log n$ and $O(\log n)$. An interesting further question here is if the cases always split into exactly these three orders, or is there a different behaviour for some $t \in [k+1,2(k-1)!]$?

We again have the trivial cases $t=1,2$. Since Theorem \ref{ES lower} was for general $k$ that result is still giving us the lower bound when $t = 3$.

It is a direct consequence of a result of Spencer \cite{spencer: first element} that $f_k(n,k)= O(\log\log n)$. In fact Spencer proved the stronger claim that there exists a family $\mathcal{F}$ of permutations from $S_n$ with size $O(\log\log n)$ such that for every $k$-tuple $X$, and every $x \in X$, there is some $P \in \mathcal{F}$ with $x <_P y$ for all $y \in X \setminus x$. That is, not only does any $k$-tuple appear in at least $k$ orders, but each element in the $k$-tuple appears first in at least one order.

That leaves us with only the following result left to prove Theorem \ref{class t k}.

\begin{thm}
For fixed $k$ and when $n$ is large, we have that $f_k(n,t) = \Theta(\log n)$ whenever $t > \frac{2(k!)}{k}$.
\end{thm}
\begin{proof}
Let $\mathcal{S}$ be a family that $(2(k-1)!+1)$-shatters every $k$-tuple. Consider $k$-tuples of the form $X :=\{ x,y,n-k+3,\dotsc,n-1,n \}$ for any $x,y \in [n-k+2]$.

Note that $x$ and $y$ must be split by at least one of $\{ n-k+3,\dotsc,n \}$ in some $P\in \mathcal{S}$. Indeed, there are only $2(k-1)!$ ways to order $X$ such that $x$ and $y$ are consecutive, yet we know that $X$ appears in at least  $2(k-1)!+1$ orders in $\mathcal{S}$.

Consider the following sets 
\begin{equation*}
    A_x^i := \{ P \in \mathcal{S} : x \text{ appears after }i \}
\end{equation*}
where $i \in [n-k+3,n]$. Then for any $x$ and $y$ there exists an $i \in [n-k+3,n]$ such that $A_x^i \neq A_y^i$.

For each $i \in [n-k+3,n]$ we define a partition of $[n-k+2]$ into at most $m$ parts, where elements $x,y \in [n-k+2]$ are in the same part if and only if $A^i_x=A^i_y$. Label the parts arbitrarily with labels $B^i_1,\dots,B^i_m$ noting that some labels may not be used at all.
Then we must have that
\begin{equation*}
    m^{k-2} \geq n-k+2.
\end{equation*}
Indeed, we are able to write each element $x$ uniquely as a $k-2$ length string from $[m]$, $x=(x_1,x_2,\dotsc, x_{k-2})$ where $x_i = r$ if $x \in B^i_r$. To see that this does create a unique identification, consider a pair $x, y$ with the same string. We must have that $x$ and $y$ are in the same $B^i$ part for all $i$, then from the definition of $B^is$ that means $A^i_y=A^i_x$ for all $i$. We have already established that distinct $x,y$ must have $A^i_y\neq A^i_x$ for some $i$ so conclude that $y=x$.

By choosing the smallest possible $m$ we can assume that there is some $i$ such that the partition has exactly $m$ parts, that is
\begin{equation*}
    |\{ B^i_1,\dots,B^i_m\}| = m.
\end{equation*}
Notice that the set 
\begin{equation*}
    \{ A^i_x : x \in [n-k+2] \}
\end{equation*}
must also have size $m$ since each $x \in B^i_r$ gives rise to the same set $A^i_x$. Therefore this set has size at least $(n-k+2)^{\frac{1}{k-2}}$ by our above bound on $m$ and hence
\begin{equation*}
    (n-k+2)^{\frac{1}{k-2}} \leq 2^{|\mathcal{S}|}.
\end{equation*}
Giving us the result
\begin{equation*}
    |\mathcal{S}| \geq \frac{1}{k-2}\log(n-k+2).
\end{equation*}
\end{proof}

\section[Totally shattering a fraction of all \texorpdfstring{$k$}{k}-tuples]{Totally shattering a fraction of all \texorpdfstring{$\boldsymbol{k}$}{k}-tuples}\label{Sec: frac}
For this problem we have a fixed number of permutations and wish to know the largest proportion of $k$-tuples that can be shattered. Recall that $F_k(n,m)$ is the maximum proportion $\alpha$ such that there is a family of size $m$ which completely shatters $\alpha\binom{n}{k}$ $k$-tuples from $[n]$.

The function $F_k(n,m)$ is weakly decreasing in $n$ when $k$ and $m$ are fixed such that $m\geq k!$.

\begin{repthm}{frac: weakly decreasing}
For fixed $k$ and fixed $m\geq k!$ we have $F_k(n,m) \geq F_k(n+1,m)$.
\end{repthm}
\begin{proof}
Suppose $F_k(n,m) = \alpha$ and consider a family $\mathcal{S}$ of $m$ permutations from $S_{n+1}$. Let $X \subseteq [n+1]$ with $|X|=n$, then by considering the permutations of $\mathcal{S}$ restricted to the elements of $X$, we see there are at most $\alpha \binom{n}{k}$ $k$-tuples shattered. In other words, at least $(1-\alpha)\binom{n}{k}$ $k$-tuples from $X$ remain un-shattered by $\mathcal{S}$. Since this is true for any such $X$ we get that the number of un-shattered $k$-tuples in $[n+1]$ is at least
\begin{equation*}
    \frac{(1-\alpha)\binom{n}{k}\binom{n+1}{n}}{\binom{n+1-k}{n-k}} =
    (1-\alpha)\binom{n+1}{k}.
\end{equation*}

Therefore the number of $k$-tuples that are shattered by $\mathcal{S}$ is at most $\alpha\binom{n+1}{k}$. Therefore, $F_k(n+1,m) \leq \alpha$.

\end{proof}

It is an easy observation that, for fixed $k$ and $m$, we can always find a suitably large value of $n$ such that every family of $m$ permutations from $S_n$ fails to shatter one $k$-tuple. Indeed, when $n \geq m^{2^m}$ we can apply the \hyperref[ES thm]{Erd\H{o}s-Szekeres Theorem} to see that there must be at least one un-shattered $k$-tuple. Therefore the limit of $F_k(n,m)$ in $n$ lies strictly between $0$ and $1$ for all $k$ and $m$, it is an interesting question to determine the value of this limit.

A direct consequence of this weakly decreasing behaviour is that the value of $F_k(N,m)$ for given fixed $N$ provides an upper bound on $F_k(n,m)$ where $n \geq N$. Consider the case when $k=3$, in particular we fix our family size at $m=6$ since this is the first non-trivial case.
\begin{thm}\label{frac:upper}
    For any $n \geq 8$ we have that $F_3(n,6) \leq \frac{11}{14}$.
\end{thm}
\begin{proof}
    We use the fact that a family of $6$ permutations of $[5]$ can shatter at most $8$ triples, so $F_3(5,6)=\frac{4}{5}$, which we have checked by hand (see Appendix \ref{F=4/5}). Applying Theorem \ref{frac: weakly decreasing} with $F_3(5,6)=\frac{4}{5}$ gives us $F_3(n,6) \leq \frac{4}{5}$ for any $n \geq 5$.

    Note that the number of triples shattered must be an integer and is given by $F_3(n,6)  \binom{n}{3}$. So for any $N$ with $F_3(N,6) \leq \alpha$ where $\alpha  \binom{N}{3}$ is not an integer we have at most $\lfloor \alpha  \binom{N}{3} \rfloor$ shattered triples. Using this and the weakly decreasing property we get a slightly lower upper bound on any $n>N$.

    In this case, when $n=8$ we have $F_3(8,6) \leq \frac{4}{5}$ and the maximum number of triples shattered by $6$ permutations is given by $F_3(8,6) \binom{8}{3} \leq \frac{4}{5} \times 56 = 44.8$, then we must shatter at most $44$ out of a possible $56$ triples, meaning $F_3(8,6) \leq \frac{44}{56}=\frac{11}{14} < \frac{4}{5}$. This along with Theorem \ref{frac: weakly decreasing} gives the desired result.

    We note that we can repeat this rounding down argument indefinitely for a smaller upper bound. The next step gives $F_3(n,6) \leq \frac{47}{60}$  whenever $n \geq 10$ by observing that $\frac{11}{14}  \binom{10}{3}$ is not an integer. In fact, let $\alpha_5=\frac{4}{5}$, $\alpha_8=\frac{11}{14}$, and $\alpha_{10}=\frac{47}{60}$ be the fractions from the first three steps in this process, and suppose $\alpha_x$ is the proportion given in $i$th step. Then the fraction given in the $(i+1)$th step will be one of $\alpha_{x+1}$, $\alpha_{x+2}$, or $\alpha_{x+3}$. In other words, for at least one of $N=x+1,x+2,x+3$ we have $\alpha_x  \binom{N}{3}$ is not an integer. This argument therefore continues indefinitely, however the actual bound it gives is hard to pin down and the numerical improvement is small and does not provide additional context, so we leave it at this.
\end{proof}

All our upper bounds come from analysing small $n$ and the above rounding argument. To improve these significantly seems to require a different and less case based approach.

To find a lower bound for $F_k(n,m)$ we show that some initial family on small $n$ can used in an iterative process, giving a family on $n^r$ which preserves much of the shattering from the initial family. Specific lower bounds can then be given by choosing a suitable initial family which shatters a high proportion of $k$-tuples.

Recall from the proof of Lemma \ref{little construction} a method of upscaling permutations on $N$ to permutations on $N^r$ (for any integer $r\geq 1$) known as \hyperref[type 1]{Type 1} permutations, we will use this method again in the next result.
\begin{thm} \label{frac: lower bound alpha family}
Suppose $\mathcal{S}$ is a family which shatters $\alpha \binom{N}{k}$ $k$-tuples from $[N]$, then for any integer $r\geq 1$, the set of \hyperref[type 1]{Type 1} permutations $\mathcal{S}^r= \{P^r : P \in \mathcal{S}\}$ shatters at least
\begin{equation}\label{frac equation: statement 1}
    \alpha \binom{N}{k} N^{(r-1)k}\frac{1-N^{r(1-k)}}{1-N^{1-k}}
\end{equation}
$k$-tuples from $[N^r]$. Hence, for all $n \geq N$
\begin{equation}\label{frac equation: statement 2}
    F_k(n,|\mathcal{S}|) \geq  \frac{\alpha(N-1)!}{(N-k)!(N^{k-1}-1)}.
\end{equation}
\end{thm}
\begin{proof}
Assign to each $x \in [N^r]$ a code $(x_1,x_2,\dotsc,x_r)\in [N]^r$ by equating the standard order on $[N^r]$ with the lexicographic order on $[N]^r$ (just as in the proof of Lemma \ref{little construction}). 

Let $P \in \mathcal{S}$ and recall that $P^r$ is the permutation from $S_{N^r}$ where $x <_{P^r} y$ if and only if $x_i <_P y_i$ where $i=d(x,y)=\min\{i:x_i \neq y_i\}$.

Note that any $k$-tuple whose elements have unique codes which differ for the first time in the same coordinate $i$, is shattered by $\mathcal{S}^r= \{P^r : P \in \mathcal{S}\}$ if the $k$-tuple of $i$th coordinates is shattered by $\mathcal{S}$. Indeed, consider $a_1,\dotsc,a_k \in [N^r]$ and write $(a^j_1,\dotsc,a^j_r)$ for the unique code of $a_j$ for each $j \in [k]$. If there exists a coordinate $i \in [r]$ such that $a^1_\ell=a^2_\ell=\dots=a^k_\ell$ whenever $\ell< i$ and where $a^1_i,a^2_i,\dotsc,a^k_i$ are all distinct, then the $k$-tuple $\{a_1,\dotsc,a_k\}$
will have its order in $P^r$ defined by the order of the $\{a^1_i,a^2_i,\dotsc,a^k_i\}$ in $P$. Hence such triples where $\{a^1_i,a^2_i,\dotsc,a^k_i\}$ is shattered by $\mathcal{S}$ are shattered by $\mathcal{S}^r$.

Therefore, we may count the minimum number of shattered $k$-tuples by counting exactly those $k$-tuples whose codes have the above property.
The number of such $k$-tuples is given by
\begin{equation*}
    \sum_{i=1}^r N^{i-1} \alpha\binom{N}{k} (N^{r-i})^k = \alpha\binom{N}{k}N^{(r-1)k}\sum_{i=0}^{r-1} \left(N^{1-k}\right)^i.
\end{equation*}
This proves the first statement.

The second statement follows by selecting an integer $r$ such that $n \leq N^r$. We will show that $F_k(N^r,|\mathcal{S}|)$ follows the statement whenever $r \geq 1$, then the decreasing property of Theorem \ref{frac: weakly decreasing} gives the statement for $F_k(n,|\mathcal{S}|)$.

Observe that equation (\ref{frac equation: statement 1}) gives us the minimum number of $k$-tuples shattered, to get an expression for the proportion $F_k(N^r,|\mathcal{S}|)$ we simply divide by the total number of $k$-tuples.
\begin{align*}
    F_k(N^r,|\mathcal{S}|) &\geq \alpha \binom{N}{k} \binom{N^r}{k}^{-1} N^{(r-1)k}\frac{1-N^{r(1-k)}}{1-N^{1-k}} \\
    &= \frac{N^{rk}(1-N^{r(1-k)})}{N^r(N^r -1)\dotsm (N^r -k +1)} \times \frac{\alpha(N-1)!}{(N-k)!(N^{k-1}-1)}
\end{align*}
To prove the second statement of the theorem, it is enough to show that
\begin{equation}\label{frac equation: fraction}
    \frac{N^{rk}(1-N^{r(1-k)})}{N^r(N^r -1)\dotsm (N^r -k +1)} \geq 1.
\end{equation}
To do this we can use induction on $k$. Note that when $k=2$ the left hand side of (\ref{frac equation: fraction}) is equal to $1$, so the statement holds for $k=2$. Then note the difference in the expression when considering $k+1$
\begin{equation*}
    \frac{N^{r(k+1)}(1-N^{r(1-(k+1))})}{N^r(N^r -1)\dotsm (N^r -(k+1) +1)} = \frac{N^{rk}(1-N^{r(1-k)})}{N^r(N^r -1)\dotsm (N^r -k +1)} \times \frac{N^{rk}-1}{(N^r-k)(N^{r(k-1)}-1)}.
\end{equation*}
Then (\ref{frac equation: fraction}) must hold for all $k>2$ as long as
\begin{equation*}
    \frac{N^{rk}-1}{(N^r-k)(N^{r(k-1)}-1)} \geq 1.
\end{equation*}
Observe that this holds whenever $kN^{r(k-1)}+N^r \geq k+1$, and since $N^r > 1$ we have satisfied all the conditions.
\end{proof}

We know from Levenshtein \cite{perfect family} that a perfect family $\mathcal{Q}_k(k+1)$ exists. By setting $\mathcal{S}= \mathcal{Q}_k(k+1)$, we can apply Theorem \ref{frac: lower bound alpha family} with $n=k+1$ and $\alpha=1$ to get the following result.
\begin{cor}\label{frac: lower bound - L perfect family iterated}
For any $n$ we have that
\begin{equation*}
    F_k(n,k!) \geq 
    \frac{k!}{(k+1)^{k-1}-1}.
\end{equation*}
\end{cor}

When $k=3$ it is known that the largest value $n$ for which a perfect family $\mathcal{Q}_k(n)$ exists is $4$. Indeed, the family $\mathcal{Q}_3(4)$ (\hyperlink{Q3(4)}{Section \ref*{intro pt 1}}) is perfect, but $F_3(5,3!)=\frac{4}{5}$ (see Appendix \ref{F=4/5}) and $F_k(n,m)$ is decreasing (Theorem \ref{frac: weakly decreasing}), meaning that $\mathcal{Q}_3(n)$ does not exists whenever $n \geq 5$. However, in general it is not known which values (if any) of $n>k+1$ admit a perfect family for any given $k$. This is an interesting open question in its own right, but perfect families for large values of $n$ would give better lower bounds for $F_k(N,k!)$ for all $N<n$.

The family $\mathcal{S}$ used in Theorem \ref{frac: lower bound alpha family} need not be perfect to give a lower bound. Note that the bound of Corollary \ref{frac: lower bound - L perfect family iterated} when $k=3$ gives $F_3(n,6) \geq \frac{2}{5}$ for all $n$. We can improve this bound by using a non-perfect family on $8$ points for $\mathcal{S}$.

\begin{cor}\label{frac: lower bound 8 family}
    For all $n$ we have $F_3(n,6) \geq \frac{17}{42}$.
\end{cor}
\begin{proof}
Apply Theorem \ref{frac: lower bound alpha family} with $n=8$, $k=3$, and $\mathcal{S}$ given by the $6$ permutations below:
\begin{equation*}
    \begin{array}{l l}
        P_1= (6,1,2,5,8,3,4,7) \quad& P_2= (5,2,1,6,7,4,3,8) \\
        P_3= (7,3,6,2,8,4,5,1) \quad& P_4= (3,7,1,5,4,8,2,6) \\
        P_5= (8,4,5,1,7,3,6,2) \quad& P_6= (4,8,2,6,3,7,1,5).
    \end{array}
\end{equation*}
Note that $\mathcal{S}$ shatters $34$ out of $56$ possible triples.
\end{proof}

Combining the results of Theorem \ref{frac:upper} and Corollary \ref{frac: lower bound 8 family} gives Theorem \ref{frac: 3 6}, and we see that $\lim_{n \to \infty} F_3(n,6)$ lies in the interval $[\frac{17}{42}, \frac{11}{14})$.

\section{Constructions for total Shattering}
Although we have a probabilistic argument showing that $f_k(n) = O(\log n)$, there is only a construction matching this size for $k=3$. We have not been able to extend any of the separation ideas that work in the special $k=3$ case to give an $O(\log n)$ construction for arbitrary values of $k$. Therefore finding such a construction is still an open problem.
This section gives an iterative construction for small shattering families and applies to any value of $k$. The main idea is to identify each element of $[n]$ with a point in the  $k$-dimensional lattice, and take permutations by grouping the points in each of the $k$ directions. Unfortunately the bound this gives is a power of $\log n$ rather than the known $O(\log n)$ bound.

\begin{lem} \label{k cube upper}
Given a family that shatters every $k$-tuple in $[n^{k-1}]$ and has size $S$, we can give an explicit construction of a family that shatters every $k$-tuple from $[n^k]$ and has size $kS$.
\end{lem}
\begin{proof}
Let $\mathcal{S}$ be a shattering family for $[n^{k-1}]$, we use the permutations in this family to construct permutations of $[n^k]$. For simplicity, let $m = n^{k-1}$.

We consider the elements of $[n^k]$ viewed geometrically as points on an integer lattice of dimension $k$. Label each element $x \in [n^k]$ by the string $(i_1,\dotsc,i_k)$ with all $i_j \in [n]$. Without loss of generality, we may assume that all elements of $[m]$ are labelled with $(1,i_2,\dotsc,i_{k})$ and hence can be thought of instead as the $k-1$ string $(i_2,\dotsc,i_{k})$. This means every $k-1$ string is associated to an element of $[m]$.

Let $r_j(x)= (i_1,\dotsc, i_{j-1},i_{j+1},\dotsc,i_k)$ be $x$ with the $j$th coordinate omitted. Note that $r_j(x)$ is therefore associated to an element of $[m]$.

For each $P\in \mathcal{S}$ we will create $k$ permutations of $[n^k]$, $P'_1,\dotsc,P'_k$. To generate the permutation $P'_j$ order $x$ before $y$ if and only if $r_j(x)$ appears before $r_j(y)$ in $P$, if $r_j(x)=r_j(y)$ then order arbitrarily.

This generates $k|\mathcal{S}|$ permutations of $[n^k]$, it is left to show that these are sufficient to shatter all $k$-tuples.

Given $k$ points in $[n]^k$, $A=\{ a_1,\dotsc,a_k \}$, there exists a direction $j \in [k]$ such that the projection of $A$ in direction $j$ has $k$ points.

Indeed, suppose not for a contradiction. For all $j \in [k]$ there is some pair $a_\ell, a_t \in A$ such that $r_j(a_\ell)=r_j(a_q)$. Define a graph $G$ with vertex set $A$, and with one edge for each direction $j\in [k]$ between some pair of vertices $u,v \in A$ with $r_j(u)=r_j(v)$. By our assumption there is at least one such pair for each $j\in [k]$, if there is a choice then pick arbitrarily. Note that if $r_j(u)=r_j(v)$ for some $j$ then we cannot have $r_{\ell}(u)=r_{\ell}(v)$ for any $\ell \in [k]\setminus j$ by definition, so we can never pick the same edge more than once. Hence our graph on $k$ vertices has exactly $k$ edges, therefore $G$ must contain a cycle. Let $v_1,\dotsc,v_t$ be a cycle in $G$, then the edge $v_1v_2$ demonstrates a change in one coordinate, say $j_1$. Similarly edge $v_2v_3$ demonstrates a change in coordinate $j_2$. Observe that $j_1$ and $j_2$ are distinct since there is only one edge for each direction. Continuing, we find $t$ distinct directions $j_1,\dots,j_t$ where $j_t$ is the direction of the edge $v_tv_1$. This is equivalent to starting with $v_1$, changing $t$ different coordinates and ending up back at $v_1$. Clearly this cannot happen and therefore we have a contradiction. 

We have shown that for any $k$-tuple $A=\{ a_1,\dotsc,a_k \}$, there is a coordinate $j$ such that $r_j(a_\ell)\neq r_j(a_q)$ for all $\ell,q \in [k]$. Hence $A$ is shattered by the collection of permutations of the form $P'_j$ where $P\in \mathcal{S}$.

Therefore our collection of $P'$s does indeed shatter all the $k$-tuples in $[n^k]$, and we used $k|\mathcal{S}| $ permutations in total.
\end{proof}

Repeatedly applying the construction in \ref{k cube upper} gives an upper bound $f_k(n) \leq (\log n)^{c_k}$ where $c_k \approx \frac{\log k}{\log k-\log(k-1)}$.

\section{Open Problems}
For the partial shattering variant, does the size of $f_k(n,t)$ always fall into one of the three sizes as classified in Theorem \ref{class t 3} and \ref{class t k}? We know that there are values of $t$ that put $f_k(n,t)$ in each of these regimes, but is there another size bracket in between $\log\log n$ and $\log n$? In particular we ask:
\begin{question}
When $k > 3$, what is the value of $f_k(n,t)$ for $k+1 \leq t \leq 2(k-1)!$?
\end{question}
It is not difficult to see that when $t$ is odd we can bound $f_k(n,t+1)\leq 2f_k(n,t)$ by taking the family that realises $f_k(n,t)$ along with all its reverse permutations. This means that for odd $k$ we know $f_k(n,k+1)$ is $O(\log\log n)$. We also ask the slightly weaker question:
\begin{question}
Is it true that $f_k(n,t)$ is one of three sizes $\Theta(\log n)$, $\Theta(\log\log n)$, or constant for all values of $t$?
\end{question}

Thinking about fractional shattering, where $\alpha\binom{n}{k}$ $k$-tuples are completely shattered, we ask:
\begin{question}\label{Q:fract general limit}
For fixed $k$ and $m$, what is the limit of $F_k(n,m)$ as $n$ increases?
\end{question}
We saw that $F_k(n,m)$ is decreasing, and by choosing $m \geq k!$ we know that the limit as $n$ increases exists and is strictly between $0$ and $1$. Progress on this question seems to require a method which does not just rely on using fixed small $n$. It would be interesting to find another method for finding upper bounds on $F_k(n,m)$ using an alternative approach. The first interesting case of Question \ref{Q:fract general limit} is the following:
\begin{question}
What is $\lim_{n \to \infty} F_3(n,6)$?
\end{question}
We saw in Section \ref{Sec: frac} that we can use perfect families to give lower bounds for $F_k(n,k!)$. Knowing more about when perfect families occur is not only useful for bounds on Fractional Shattering, but is also interesting and worthwhile in its own right. We therefore highlight the question:
\begin{question}
What is the largest value of $n>k+1$ for which the perfect family $\mathcal{Q}_k(n)$ exists?
\end{question}

\begin{appendices}

\section{ }\label{F=4/5}
To show that $F_3(5,6)=\frac{4}{5}$ we must show that any $6$ permutations from $S_5$ shatter at most $8$ triples out of a possible $10$. Therefore it is sufficient to show that in any $6$ permutations there are at least $2$ distinct triples that are not shattered, and hence have a repeated pattern.

So we look for a family of $6$ permutations from $S_5$ with at most $1$ triple un-shattered (i.e. with a pattern repeated), if we cannot find such a family then we have proved the statement.

If such a family exists, we may assume that the triple that is not shattered contains the element $5$. From this we see that the family of $6$ permutations of $S_4$ generated by omitting $5$ must shatter every triple from $[4]$. Since the family $\mathcal{Q}_k(4)$ is unique up to isomorphism, to prove the statement we show that 
adding $5$ in any position to permutations from $\mathcal{Q}_k(4)$ results in a family that leaves $2$ triples un-shattered.

Here is the family $\mathcal{Q}_k(4)$:
\begin{equation*}
    \begin{array}{ l l l}
        Q_1= (1,2,3,4) \quad& Q_2= (2,4,1,3) \quad& Q_3= (3,4,1,2) \\
        Q_4= (1,4,3,2) \quad& Q_5= (4,2,3,1) \quad& Q_6= (3,2,1,4).
    \end{array}
\end{equation*}
The permutation generated by adding $5$ into $Q_i$ in some position will be denoted $Q'_i$.

We split into $5$ cases, one for each location of element $5$ in $Q_1$.

Case 1: $Q'_1=(1,2,3,4,5)$.

So we have fixed $Q'_1$, consider the options for $Q'_2$.

If $Q'_2 = (2,4,1,3,5)$ the triples $\{1,3,5\}$ and $\{2,4,5\}$ appear in the same order in both $Q'_1$ and $Q'_2$. Any family that contains $Q'_1$ and $Q'_2 = (2,4,1,3,5)$ is not the family we search for, so we look at a different option for $Q'_2$.

If $Q'_2 = (2,4,1,5,3)$ then the triple $\{2,4,5\}$ appears in the same order in $Q'_1$ and $Q'_2$. Consider now $Q'_3 = (3,4,\star,1,\star,2,\star)$ where $5$ appears in any location denoted by $\star$, the triple $\{3,4,5\}$ appears in the same order in $Q'_1$ and $Q'_3$. This forces both triples $\{2,4,5\}$ and $\{3,4,5\}$ to repeat. So we assume instead that $Q'_3 = (\star,3,\star,4,1,2)$.

So we have fixed $Q'_1=(1,2,3,4,5)$, $Q'_2 = (2,4,1,5,3)$, and $Q'_3 = (\star,3,\star,4,1,2)$ and we know that $\{2,4,5\}$ is repeated already. Consider $Q'_5 = (4,2,3,\star,1,\star)$, then $\{2,3,5\}$ appears in the same order in $Q'_1$ and $Q'_5$, meaning both $\{2,4,5\}$ and $\{2,3,5\}$ are not shattered. Similarly if $Q'_5 = (4,\star,2,\star,3,1)$ the triple $\{3,4,5\}$ is copied in $Q'_2$ and $Q'_5$, meaning $\{2,4,5\}$ and $\{3,4,5\}$ are not shattered. Finally, if $Q'_5 = (5,4,2,3,1)$ we have $\{1,4,5\}$ appearing in the same pattern in $Q'_3$ and $Q'_5$ giving the pair $\{2,4,5\}$ and $\{1,4,5\}$.

We record the above information as follows.
\begin{equation*}
\begin{array}{ll}
    Q'_2 = (2,4,1,5,3) & \{2,4,5\} \\
    \phantom{Q'_2 =} Q'_3 = (3,4,\star,1,\star,2,\star) & \phantom{\{2,4,5\},}\{3,4,5\} \\
    \phantom{Q'_2 =} Q'_3 = (\star,3,\star,4,1,2) & \phantom{\{2,4,5\},} \\
    \phantom{Q'_2 =} \phantom{Q'_3 =} Q'_5 = (4,2,3,\star,1,\star) & \phantom{\{2,4,5\},}\{2,3,5\} \\
    \phantom{Q'_2 =} \phantom{Q'_3 =} Q'_5 = (4,\star,2,\star,3,1) & \phantom{\{2,4,5\},}\{3,4,5\} \\
    \phantom{Q'_2 =} \phantom{Q'_3 =} Q'_5 = (5,4,2,3,1) & \phantom{\{2,4,5\},}\{1,4,5\}
\end{array}
\end{equation*}
This means that the family we search for cannot contain $Q'_1$ and $Q'_2 = (2,4,1,5,3)$, so we assume next that $Q'_2= (2,4,5,1,3)$. It happens that for $Q'_2= (2,4,5,1,3)$ the case is the same as that of $Q'_2 = (2,4,1,5,3)$.

For the remaining two options for $Q'_2$ we have the following.
\begingroup
\allowdisplaybreaks
\begin{alignat*}{2}
    & Q'_2 = (2,5,4,1,3) &&  \\
    &\phantom{Q'_2 =} Q'_4 = (1,4,3,\star,2,\star) && \{1,4,5\},\{1,3,5\} \\
    &\phantom{Q'_2 =} Q'_4 = (1,4,5,3,2) && \{1,4,5\} \\
    &\phantom{Q'_2 =} \phantom{Q'_4 =} Q'_5 = (4,2,3,\star,1,\star) && \phantom{\{1,4,5\},}\{2,3,5\} \\
    &\phantom{Q'_2 =} \phantom{Q'_4 =} Q'_5 = (\star,4,\star,2,\star,3,1) && \phantom{\{1,4,5\},}\{3,4,5\} \\
    &\phantom{Q'_2 =} Q'_4 = (1,5,4,3,2) && \{3,4,5\} \\
    &\phantom{Q'_2 =} \phantom{Q'_4 =} Q'_3 = (3,4,1,\star,2,\star) && \phantom{\{3,4,5\},}\{1,2,5\} \\
    &\phantom{Q'_2 =} \phantom{Q'_4 =} Q'_3 = (3,4,5,1,2) && \phantom{\{3,4,5\},} \\
    &\phantom{Q'_2 =} \phantom{Q'_4 =} \phantom{Q'_3 =} Q'_5 = (4,2,\star,3,\star,1,\star) && \phantom{\{3,4,5\},}\{2,3,5\} \\
    &\phantom{Q'_2 =} \phantom{Q'_4 =} \phantom{Q'_3 =} Q'_5 = (\star,4,\star,2,3,1) && \phantom{\{3,4,5\},}\{1,4,5\} \\
    &\phantom{Q'_2 =} \phantom{Q'_4 =} Q'_3 = (\star,3,\star,4,1,2) && \phantom{\{3,4,5\},}\{1,4,5\} \\
    &\phantom{Q'_2 =} Q'_4 = (5,1,4,3,2) && \{3,4,5\},\{1,3,5\}\\
    &\phantom{Q'_4 =} && \phantom{Q'_4 =} \\
    &Q'_2 = (5,2,4,1,3) &&  \\
    &\phantom{Q'_2 =} Q'_4 = (1,4,3,\star,2,\star) && \{1,4,5\},\{1,3,5\} \\
    &\phantom{Q'_2 =} Q'_4 = (1,4,5,3,2) && \{1,4,5\} \\
    &\phantom{Q'_2 =} \phantom{Q'_4 =} Q'_5 = (4,2,3,\star,1,\star) && \phantom{\{1,4,5\},}\{2,3,5\} \\
    &\phantom{Q'_2 =} \phantom{Q'_4 =} Q'_5 = (4,2,5,3,1) && \phantom{\{1,4,5\},}\{3,4,5\} \\
    &\phantom{Q'_2 =} \phantom{Q'_4 =} Q'_5 = (\star,4,\star,2,3,1) && \phantom{\{1,4,5\},}\{1,2,5\} \\
    &\phantom{Q'_2 =} Q'_4 = (1,5,4,3,2) && \{3,4,5\} \\
    &\phantom{Q'_2 =} \phantom{Q'_4 =} Q'_5 = (4,2,3,\star,1,\star) && \phantom{\{1,4,5\},}\{2,3,5\} \\
    &\phantom{Q'_2 =} \phantom{Q'_4 =} Q'_5 = (4,2,5,3,1) && \phantom{\{1,4,5\},} \\
    &\phantom{Q'_2 =} \phantom{Q'_4 =} \phantom{Q'_5 =} Q'_3 = (3,4,1,\star,2,\star) && \phantom{\{1,4,5\},}\{1,2,5\} \\
    &\phantom{Q'_2 =} \phantom{Q'_4 =} \phantom{Q'_5 =} Q'_3 = (\star,3,\star,4,\star,1,2) && \phantom{\{1,4,5\},}\{1,4,5\} \\
    &\phantom{Q'_2 =} Q'_4 = (5,1,4,3,2) && \{3,4,5\},\{1,3,5\} \\
\end{alignat*}
\endgroup

These show that if $Q'_2= (2,5,4,1,3)$ or $Q'_2= (5,2,4,1,3)$, there is no position $5$ can take in $Q'_4$ without causing a pair of un-shattered triples. With all of these pieces we see that no matter where $5$ is found in $Q'_2$ it leads to a family with at most $8$ triples shattered.

We continue like this for the remaining cases, we omit the details since they are straightforward and unenlightening. None of the cases provide a family that shatters more than $8$ triples, and hence we have $F_3(5,6)=\frac{4}{5}$.

\end{appendices}

\end{document}